\documentclass[12pt, leqno]{article}
\usepackage{amsmath, amsthm}
\usepackage{amsfonts}
\usepackage{amssymb}
\usepackage[usenames,dvipsnames]{color}%\usepackage[dvips]{color}
\usepackage{graphicx}
\usepackage{hyperref}
\usepackage{bm}
\usepackage{listings}
\usepackage{nicefrac} % For comparison
\usepackage{xfrac}    % Works better with other fonts

\usepackage{color,calc}

\makeatletter
\definecolor{shade}{gray}{0.8}
        {
          \raggedright
        \setlength{\rightmargin}{\leftmargin}
        \setlength{\itemsep}{-12pt}
        \setlength{\parsep}{20pt}
        \begin{lrbox}{\@tempboxa}%
        \begin{minipage}{\linewidth-2\fboxsep}
        }%
        {
        \end{minipage}%
        \end{lrbox}%
        \fcolorbox{black}{shade}{\usebox{\@tempboxa}}\newline
        }%
\makeatother

\newtheorem{theorem}{Theorem}
\newtheorem{lemma}{Lemma}
\newtheorem{cor}{Corollary}
\renewcommand{\d}{{\rm d}}

\newcommand{\PP}{\mathbb{P}^\eta}
\newcommand{\EE}{\mathbb{E}^\eta}

\renewcommand{\eqref}[1]{\hyperref[#1]{(\ref*{#1})}}

\DeclareMathOperator{\sgn}{sgn}

\newcommand{\dd}{\mathrm{d}}

\newcommand*{\pref}[1]{\hyperref[#1]{(\ref*{#1})}}
\newcommand*{\refpref}[2]{\hyperref[#2]{\ref*{#1}(\ref*{#2})}}

\newcommand{\D}{{\rm d}}

\newcommand{\R}{\mathbb{R}}%Reales
\title{ The Doob--McKean  identity  for  stable L\'evy  processes}

\author{  Andreas E. Kyprianou\thanks{Department of Mathematical Sciences, University of Bath, Claverton Down, Bath, BA2 7AY, UK. Email: \texttt{a.kyprianou@bath.ac.uk}}, \ Neil O'Connell\thanks{School of Mathematics and Statistics, University College Dublin, Dublin 4,
Ireland.
E-mail: \texttt{neil.oconnell@ucd.ie}
}
$^,$\thanks{Research supported by the European Research Council (Grant No. 669306).}
\\
{\it  On the occasion of Ron Doney's 80th birthday}
}  
\begin{document}

\maketitle

\begin{abstract}
\noindent We re-examine the celebrated Doob--McKean identity that identifies a conditioned one-dimensional Brownian motion 
as the radial part of a 3-dimensional Brownian motion or, equivalently, a Bessel-3 process, albeit now in the analogous setting of isotropic $\alpha$-stable processes. We find a natural analogue that matches  the Brownian setting, with the role of the Brownian motion replaced by that of the isotropic $\alpha$-stable process, providing one interprets the components of the original identity in the right way. 

\medskip

\medskip

\end{abstract}
%\tableofcontents

\section{Introduction} A now-classical result in the theory of Markov processes due to Doob \cite{doob} and McKean \cite{mckean} equates the law of a Brownian motion conditioned to stay positive with that of a Bessel-3 process; see also  \cite{williams1, pitman, williams2}. A precise statement of this identity can be made in a number of different ways as each of the two processes that are equal in law have several different representations. For the purpose of this exposition, it is worth reminding ourselves of them.

%\subsection{Brownian motion conditioned to stay postiive}
Denote by $\mathbb{D}(\mathbb{R})$ the space of c\`adl\`ag paths  $\omega : [0,\infty) \to \mathbb{R}\cup \Delta$ with lifetime $\zeta = \inf\{t>0 : \omega_t  =\Delta\}$, where $\Delta$ is a cemetery state.  The space $\mathbb{D}(\mathbb{R})$ will be equipped with the Skorokhod topology and its natural Borel $\sigma$-algebra into which is embedded the natural filtration $(\mathcal{F}_s,s\ge 0)$. On this space, we will denote by  $B = (B_t, t\geq 0)$ the coordinate process whose  probabilities $\mathbb{P} = (\mathbb{P}_x, x\in\mathbb{R})$ are those of  a standard one dimensional Brownian motion. For each $t\geq 0$, $x>0$, the limit 
\begin{equation}
\mathbb{P}^\uparrow_x(A,\, t<\zeta): = \lim_{\varepsilon\to0}\mathbb{P}_x (A, t<\mathbf{e}/\varepsilon \, |\, \tau^-_0(B)>\mathbf{e}/\varepsilon ),
\label{condition}
\end{equation}
where $\mathbf{e}$ is an independent exponentially distributed random variable with unit mean, $\tau^-_0(B) = \inf\{t>0: B_t<0\}$, defines a new family of probabilities on $\mathbb{D}(\mathbb{R}_{\geq 0}) := \{\omega \in\mathbb{D}(\mathbb{R}) : \omega \in (0,\infty)\cup\Delta\}$. It turns out that $\mathbb{P}^\uparrow = (\mathbb{P}^\uparrow_x, x>0)$ defines a conservative (i.e. $\zeta = \infty$) Markov process on $[0,\infty)$. As such, $(B,\mathbb{P}^\uparrow)$ is the  sense in which we can understand  Brownian motion conditioned to stay positive.

\medskip

Thanks to the well known fact that the probability $\mathbb{P}_x (\tau^-_0(B) >t )\sim  x/\sqrt{2\pi t}$, as $t\to\infty$, it is easy to verify by taking its Laplace transform followed by an integration by parts, then an application of  the classical Tauberian Theorem, that, up to an  constant $c>0$,  $\mathbb{P}_x (\tau^-_0(B) >\mathbf{e}/\varepsilon)\sim  cx\sqrt{\varepsilon}$.
One thus easily verifies  
from \eqref{condition}, with the help of an easy dominated convergence argument,  that $(B,\mathbb{P}^\uparrow)$ satisfies 
\begin{equation}
\left. \frac{\D \mathbb{P}^\uparrow_x}{\D \mathbb{P}_x}\right|_{\mathcal{F}_t}   =\frac{B_t}{x}\mathbf{1}_{\{t<\tau^-_0(B)\}}, \qquad x, t>0.
\label{COM}
\end{equation}
The change of measure \eqref{COM} presents a second definition of the Brownian motion conditioned to stay positive via a Doob $h$-transform with respect to Brownian motion killed on exiting $[0,\infty)$, using the harmonic function $h(x)= x$. Suppose we write $p_t(x,y)$ and $p^\dagger_t(x, y)$, $t\geq0$, $x,y>0$, for the transition density of Brownian motion and of Brownian motion killed on exiting $[0,\infty)$, respectively. Then another way of expressing \eqref{COM} is via the harmonic transformation
\begin{equation}
p^\uparrow(x, y) := \frac{y}{x}p_t^\dagger(x,\D y) = \frac{y}{x} (p_t(x,y) - p_t(x,-y)) , \qquad x,y>0.
\label{transitions}
\end{equation}

\medskip
As alluded to above, the so-called {\it Doob--McKean identity} states that the process $(B,\mathbb{P}^\uparrow)$ is equal in law to a Bessel-3 process.  
%Suppose now that $(R_t,t\geq 0)$, defined in the usual way on $\mathbb{D}(\mathbb{R}_{\geq 0})$ as above,  with probabilities $(\mathbf{P}_x, x>0)$ is a Bessel-3 process.
 There are also several ways that one may define the latter processes. Among the many,  there are three that we mention here. 
\medskip

 As a parametric family indexed by $\nu\geq0$, Bessel-$\nu$ processes are defined as non-negative valued, conservative, one-dimensional diffusions which can be identified via the action of their generator $L^\nu$, which satisfies
\begin{equation}
\label{generator}
L^\nu  = \frac{1}{2}\left(\frac{\D^2}{\D x^2} + \frac{\nu -1}{x}\frac{\D}{\D x}\right), \qquad x>0, 
\end{equation}
such that the point $0$ is treated as an absorbing boundary if $\nu = 0$, as a reflecting boundary if $\nu \in(0,2)$
 and as an entrance boundary if $\nu\geq 2$.  As such, the associated transition density can be identified as a non-zero solution to the backward equation given by $L^\nu$.  In general, the 
 transition density can be identified explicitly with the help of Bessel functions (hence the name of the family of processes).
 In the special case that  $\nu = 3$, it turns out that the transition density  can be more simply identified  by the right-hand side of \eqref{transitions}.  
 \medskip

In the setting that $\nu$ is a natural number, in particular, in the case that $\nu = 3$, the generator \eqref{generator} is also the radial component of the $\nu$-dimensional Laplacian. Noting that the latter is the generator of a $\nu$-dimensional Brownian motion, we also see that, for positive integer values of $\nu$, the Bessel-$\nu$ process is also the radial distance from the origin of a $\nu$-dimensional Brownian motion; cf. \cite{KMcG}.  This also illuminates the need for the point $0$ to be either reflecting or an entrance point when $\nu>0$, at least for $\nu\in\mathbb{N}$.
 
%Bessel-$\nu$ processes are also continuous self-similar Markov processes with scaling index equal to 2. This is a fact that can be deduced  the transition density, say $p^\nu_t(x,y)$, $t\geq0$, $x,y\geq0$, in the sense that 
%\[
%cq^\nu_t(cx, cy) =q^\nu_{c^{-2}t}(x,y), \qquad x,y\geq 0, t>0,
%\]
%thanks to the same scaling holding for $p_t(x,y)$. Suppose now that $(R_t,t\geq 0)$, defined in the usual way on $\mathbb{D}(\mathbb{R}_{\geq 0})$ as above,  with probabilities $(\mathbf{P}_x, x>0)$ is a Bessel-$\nu$ process. The self-similarity property in this setting means that, for all $c,x>0$,
%\begin{equation}
%(cR_{c^{-2}t}, t\geq 0) \text{ under } {\mathbf P}_x \text{ is equal in law to  }(R_{t}, t\geq 0) \text{ under $\mathbf{P}_{cx}$.}
%\end{equation}
%As a continuous self-similar Markov process with index $2$, until first hitting the origin (if at all), a Bessel-$\nu$ process can also be represented in the form 
%\begin{equation}
%R_t = {\rm e}^{\xi_{\varphi( t) } }, \qquad t\leq \int_0^\infty {\rm e}^{2\xi_u}\D u,
%\label{pssMp1}
%\end{equation}
%where $\xi_{t} = B_t + (\nu -2)t/2$, $t\geq0$,   and 
%\[
%\varphi(t) = \inf\{s>0: \int_0^s {\rm e}^{2\xi_u}\D u>t\}, 
%\]
%with the usual convention that $\inf\emptyset := \infty$.
%

\medskip

The Doob--McKean identity is present-day nested in a much bigger dialogue concerning the representation of conditioned, path-segment-sampled and time-reversed stochastic processes, including general diffusions, random walks and L\'evy processes; see e.g. \cite{doob, mckean, williams1, pitman, williams2, bertoinbook, RW2, BD, chaumont, CD1, CD2} and others. In this article we add to the list of extensions to the Doob--McKean identity by looking at the setting in which the role of the Brownian motion is replaced by an isotropic  $\alpha$-stable  process. 

\medskip

%{\color{red}Remarks on proof and how it gives meaning to the Elliot-Feller paper...}

\section{Doob-McKean for isotropic $\alpha$-stable processes}
We recall that an isotropic $\alpha$-stable process (henceforth sometimes referred to as a stable process or a symmetric stable process in one dimension) in dimension $d\in\mathbb{N}$, with coordinate process say $X = (X_t,t\geq 0)$ and probabilities $\texttt{P}^{\alpha,d} = (\texttt{P}^{\alpha, d}_x, x \in \mathbb{R}^d)$,   is a L\'evy process which is also a self-similar Markov process, which has self-similarity index $\alpha$. More precisely, as a L\'evy process, its transitions are uniquely described by its characteristic exponent given by the identity
\[
\texttt{E}^{\alpha, d}_0[\exp({\rm i}\theta X_t)] = \exp(-|\theta|^\alpha t), \qquad t\geq 0,
\]
where we interpret $\theta X_t$ as an inner product in the setting that $d\geq 2$. For the pure jump case that we are interested in, it is necessary that $\alpha\in(0,2)$.
As a self-similar Markov process with index $\alpha$, it satisfies the scaling property that, for all $c>0$,
\begin{equation}
(cX_{c^{-\alpha}t}, t\geq0) \text{ under $\texttt{P}^{\alpha, d}_x$ is equal in law to }(X,\texttt{P}^{\alpha, d}_{cx}).  
\label{scalingC}
\end{equation}
In any dimension, $(X , \texttt{P}^{\alpha,d} )$ has a transition density and, for example, in the setting $d=1$, if we denote it by   $q^{(\alpha)}_t(x,y)$, $x,y\in\mathbb{R}$, then the scaling property \eqref{scalingC} manifests in the form
\begin{equation}
cq^{(\alpha)}_t(cx, cy) =q^{(\alpha)}_{c^{-\alpha}t}(x,y), \qquad x,y\geq 0, t>0.
\label{Cauchyss}
\end{equation}

We note that the Cauchy process has a symmetric distribution in one dimension and is isotropic in higher dimensions. As a L\'evy process, its jump measure is given by 
\begin{equation}
\Pi(\d z) = 2^{\alpha}\pi^{-d/2}\frac{\Gamma((d+\alpha)/2)}{\big|\Gamma(-\alpha/2)\big|}\frac{1}{|z|^{\alpha +d}}\,\d z,
\qquad z\in\mathbb{R}^d
% CAUCHY CASE%% \Pi(B) =\frac{\Gamma((d+1)/2)}{\pi^{(d+1)/2}} \int_{B}\frac{{\rm d} z}{|z|^{1 + d}} ,
 \label{otherrepscartesian}
 \end{equation}
% and in (generalised) polar coordinates, 
% \begin{equation}
%\Pi(B) = 2^{\alpha-1}\pi^{-d}\frac{\Gamma((d+\alpha)/2)\Gamma(d/2)}{\big|\Gamma(-\alpha/2)\big|}\int_{\mathbb{S}^{d-1}}\sigma_1({\rm d} \phi)\int_0^\infty\mathbf{1}_{B}(r\phi)\frac{ 1}{r^{\alpha+1}}\,\d r,
% \label{otherreps}
%\end{equation}
where $B$ is a Borel set in $\mathbb{R}^d$. A special case of interest will be when $\alpha =1$ and when $d=1$, in which case, \eqref{otherrepscartesian}
takes the form 
%\begin{equation}
%\Pi(\D x) = \frac{1}{\pi}\frac{1}{x^2}\D x , \qquad  x\in\mathbb{R}%\quad\text{ and }\quad\Pi(B) =\frac{\Gamma((d+1)/2)}{\pi^{(d+1)/2}} \int_{B}\frac{{\rm d} z}{|z|^{1 + d}}, \quad d>1, x\in\mathbb{R}^d.
% \label{otherrepscartesian}
%\end{equation}
%In the latter case, when $d = 1$, we have 
\[
\Pi(\D x) = \frac{1}{\pi}\frac{1}{x^2}\D x, \qquad x\in\mathbb{R}.
\]
Moreover, the transition density, more conveniently written as $(q_t, t\geq 0)$ rather than $(q^{(1)}_t, t\geq0)$, is given by 
\begin{equation}
q_t(x,y) = \frac{1}{\pi}\frac{t}{(y-x)^2 + t^2}, \qquad x,y\in \mathbb{R}, t> 0,
\label{Cauchydensity}
\end{equation}
%Moreover, more generally, when $d = 1$ and  $\alpha$ can take any value in $(0,2)$, 
%\[
% \Pi(\d x)=|x|^{-1-\alpha} \left(\Gamma(1+\alpha)  \frac{\sin(\pi \alpha /2)}{\pi} {\bf 1}_{(x>0)} +\Gamma(1+\alpha)  \frac{\sin(\pi \alpha\hat/ 2)}{\pi} {\bf 1}_{(x<0)} \right)\,\d x, \qquad x\in\mathbb{R},
%\]
from which we can verify the scaling property  \eqref{scalingC} directly. 

\medskip

Given the summary of the the  Doob--McKean identity for the Brownian setting above, the stable-process analogue we present as our main result below matches perfectly the Brownian setting providing one interprets the components in the identity in the right way. 

\begin{theorem}\label{CauchyDM}
The kernel 
\begin{align}
q^{(\alpha),*}_t(x,y) &= \frac{y}{x}\left(q^{(\alpha)}_t(x,y)- q^{(\alpha)}_t(x,-y)\right)\qquad x,y\geq0, t>0 
%\notag \\
%&= \frac{1}{\pi}\frac{4y^2t}{ (y^2-x^2)^2 +2t^2 (y^2 + x^2) + t^4 },\qquad x,y\geq0, t>0,
\label{newSG}
\end{align}
defines a conservative Feller semigroup, say $Y = (Y_t, t\geq0)$, on $[0,\infty)$ which is self-similar with index $\alpha$. Moreover,  $Y$ is equal in law to the radial part of a three-dimensional isotropic $\alpha$-stable process.
\end{theorem}

An easy corollary of the above result is the following.
\begin{cor}\label{cor}
The transition density of the radial part of a  3-dimensional Cauchy process  is given by 
\begin{equation}
q^{(1),*}_t(x,y) = \frac{1}{\pi}\frac{4y^2t}{ (y^2-x^2)^2 +2t^2 (y^2 + x^2) + t^4 },\qquad x,y\geq0, t>0.
\end{equation}
\end{cor}

\begin{proof}[Proof of Theorem \ref{CauchyDM}]
The proof is a relatively elementary consequence of the classical Doob--McKean identity once one takes account of the following basic fact; cf. e.g. Chapter 3 of \cite{KPbook}.

\begin{lemma}\label{subordinate} If $ (B^{(d)}_t, t\geq 0)$ is a standard $d$-dimensional Brownian motion ($d\geq 1$) and $\Lambda = (\Lambda_t, t\geq 0)$ is an independent  stable subordinator with index $\alpha/2$, where $\alpha\in(0,2)$, then $( \sqrt{2}B^{(d)}_{\Lambda_t}, t\geq 0)$ is an isotropic $d$-dimensional stable process with index $\alpha$. 
\end{lemma} 

An immediate consequence of Lemma \ref{subordinate} is that, e.g. in one dimension, we can identify the semigroup of a symmetric stable process with index $\alpha$ via
\[
q_t^{(\alpha)}(x,y) =\int_{0}^\infty  \gamma^{(\alpha/2)}_t(s) \frac{1}{2^{d/2}}
p^{(d)}_s(x, y/\sqrt{2})\d s
\]
where $p^{(d)}_t(x,y)$, $x,y\in\mathbb{R}^d$ is the transition density of a standard Brownian motion in $\mathbb{R}^d$ (and for consistency we have $p^{(1)}_t = p_t$, $t\geq0$.)
\[
\gamma^{(\alpha/2)}_t(s)
= \frac{1}{\pi}\sum\limits_{n\ge 1} (-1)^{n-1} \frac{\Gamma(1+\frac{\alpha n}{2})}{n!} \sin\left(\frac{n \pi \alpha}{2}\right) 
t^n s^{-\frac{n \alpha}{2} -1}, \qquad x>0,
\]
is the transition density of the stable subordinator with index $\alpha/2$.

\medskip

Replacing $y$ by $y/\sqrt{2}$ in \eqref{transitions} and dividing through by $\sqrt{2}$, by integrating against the kernel $\gamma^{(\alpha/2)}$ we see with the help of Lemma \ref{subordinate} that  
\begin{align*}
\frac{1}{\sqrt{2}}\int_0^\infty  \gamma^{(\alpha/2)}_t(s) p^\uparrow_s(x,y/\sqrt{2})\d s
& = \frac{y}{x}\left(q^{(\alpha)}_t(x,y)- q^{(\alpha)}_t(x,-y)\right), \qquad x, y\geq 0, t\geq0.
\end{align*}
Writing $\mathbb{P}^{(3)}$ for the law of $3$-dimensional Brownian motion with coordinate process $(B^{(3)}_t, t\geq0)$ as a coordinate process on  $\mathbb{D}(\mathbb{R})$. Since $(p^\uparrow_t, t\geq0)$ is the transition density of of a Bessel-3 process, which is also the transition density of the radius of a 3-dimensional standard Brownian motion, 
we know that 
\[
\frac{1}{\sqrt{2}} p^\uparrow_s(x,y/\sqrt{2})\d y = \mathbb{P}^{(3)}_{(x,0,0)}(|\sqrt{2}B^{(3)}_t|\in \d y), \qquad y,t\geq0.
\]
As such, it follows that 
\[
\frac{1}{\sqrt{2}}\int_0^\infty  \gamma^{(\alpha/2)}_t(s) p^\uparrow_s(x,y/\sqrt{2})\d s =
 \mathbb{P}^{(3)}_{(x,0,0)}(|\sqrt{2}B^{(3)}_{\Lambda_t}|\in \d y),
\]
where $\Lambda$ is an independent stable subordinator with index $\alpha/2$. Lemma \ref{subordinate} now allows us to conclude that \eqref{newSG}
agrees with the transition semigroup of the radial component of a 3-dimensional stable process. 
On account of the fact that the radial component of an isotropic stable process is a conservative self-similar Markov process (and in particular a Feller process), we see that the semigroup in \eqref{newSG} must also offer the same properties. This also includes the existence of an entrance law at zero which is affirmed by the representation given in Lemma \ref{subordinate}.
\end{proof}

\section{The special case of Cauchy processes}
The special case of the Doob--McKean identity for $\alpha = 1$, i.e. the Cauchy process, reveals a few more details that we can explore further. In the subsections below, we look at the Doob--McKean identity in in terms of the Lamperti representation of self-similar Markov processes, its relation with the Cauchy process conditioned to stay positive and in terms of a pathwise interpretation. 

\subsection*{Lamperti representation of the Doob-McKean identity}
As a self-similar Markov process with index 1, the process $Y$ in Theorem \ref{CauchyDM} when $\alpha = 1$ enjoys a Lamperti representation. Specifically, 
\begin{equation}
\label{pssMp2}
Y_t = {\rm e}^{\xi_{\varphi(t)}}, \qquad t\leq \int_0^\infty {\rm e}^{\xi_u}\D u, 
\end{equation}
where $\varphi(t) = \inf\{s>0: \int_0^s\exp(\xi_u)\D u>t\}$ and  $(\xi_t, t\geq0) $  is a L\'evy process, which is possibly killed at an independent and exp

\medskip

Another way of understanding the statement in the second part of Theorem \ref{CauchyDM} is that the L\'evy process $\xi$ agrees with the one that underlies the Lamperti representation of the radial part of a three-dimensional Cauchy process. The reason why the latter is a positive  self-similar Markov process was examined in \cite{CPP}; see also Chapter 5 of \cite{KPbook}. Indeed, there it was shown that the radial part of a 3-dimensional Cauchy process has underlying L\'evy process, say $(\eta_t ,t\geq0)$, with probabilities $(\PP_x, x\in\mathbb{R})$,  which is identified via its characteristic exponent $\Psi(z) = -\log \int_{\mathbb{R}}{\rm e}^{{\rm i}zx}\PP_0(\eta_1\in \D x)$, where
\begin{align*}
\Psi(z)
& = 2\frac{\Gamma(\frac{1}{2}(-{\rm i}z +1 ))}{\Gamma(-\frac{1}{2}{\rm i}z)}\frac{\Gamma(\frac{1}{2}({\rm i}z +3))}{\Gamma(\frac{1}{2}({\rm i}z +2))}
%= -({\rm i}z +1)\frac{\Gamma(-\frac{1}{2}({\rm i}z +1 ))}{\Gamma(-\frac{1}{2}{\rm i}z)}\frac{\Gamma(\frac{1}{2}({\rm i}z +3))}{\Gamma(\frac{1}{2}({\rm i}z +2))}\\
%&=
% -({\rm i}z +1)\frac{\sin(\pi {\rm i}z/2)}{\cos(\pi{\rm i} z/2)} =-{\rm i} ({\rm i}z +1) \tanh (z\pi/2)\\
%&
=
(z-{\rm i})\tanh(\pi z/2), \qquad z\in\mathbb{R}.
\label{a}
\end{align*}
An equivalent way of identifying $\eta$ is as a pure jump process, with no killing (note that $\Psi(0) =0$) and  with L\'evy measure having density taking the form  
  \begin{equation}
\mu(x)=\frac{4}{\pi}
\frac{{\rm e}^{3x}}{({\rm e}^{2x}-1)^{2}}, \qquad x\in\mathbb{R}.
 \label{HGdensity}
 \end{equation}
Note that for small $|x|$ the density above behaves like $O(|x|^{-2})$, for large positive $x$, it behaves like $O({\rm e}^{-x})$ and for large negative $x$, it behaves like $O({\rm e}^{-3|x|})$. As such, the process $\eta$ has paths of unbounded variation and its law enjoys exponential moments; in particular $\eta$ has a finite first moment. 
\medskip

The long term linear growth of $\eta$ (in the sense of the Strong Law of Large Numbers) is given by the mean $\EE_0[\eta_1] = \pi /2$
which can also be computed from the value of ${\rm i}\Psi'(0)$; see also Proposition 1 of \cite{KP}. Not surprisingly this implies that $\lim_{t\to\infty}{\eta_t} = \infty$ almost surely. This is consistent with the fact that a three-dimensional Cauchy process is transient and hence, its radial component drifts to $+\infty$, which implies its underlying L\'evy process must too. Note, in the latter observation, we are also using the fact that positive self-similar Markov processes are either: Transient to infinity, corresponding to the underlying L\'evy process drifting to $+\infty$; Interval recurrent, corresponding to the underlying L\'evy process oscillating; Continuously absorbed at the origin, corresponding to the case that the underlying L\'evy process drifts to $-\infty$; Absorbed at the origin by a jump; corresponding to the case that the underlying L\'evy process is killed at an independent and exponentially distributed time. See \cite{Lamperti, Kbook, KPbook} for further details.
\medskip

Because $\eta$ has a finite first moment, we can relate \eqref{HGdensity} to \eqref{a} via the particular arrangement of the L\'evy--Khintchine formula 
\begin{equation}
\Psi(z)  = - \frac{\pi}{2}{\rm i}z+ \int_\mathbb{R}\left(1-{\rm e}^{{\rm i}z x}+ {\rm i}zx\right)\mu(x)\D x,\qquad z\in\mathbb{R}.
\label{arrangement}
\end{equation}
This arrangement will prove to be convenient in the following Corollary.

\begin{cor}\label{L2} Suppose that $\mathcal{C}^2(\mathbb{R}_{\geq 0})$ is the space of twice continuously integrable functions on $\mathbb{R}_{\geq 0}$. On $\mathcal{C}^2(\mathbb{R}_{\geq 0})$, the action of the generator $\mathcal{L}$ associated to the process $Y$ in Theorem \ref{CauchyDM} is given by 
\begin{align}
\mathcal{L}f(x)& = \frac{\pi}{2}f'(x)+\frac{4}{\pi x}\int_0^\infty\left(f(xu)-f(x)-xf'(x)\log u\right)\frac{u^2}{(u^2-1)^2}\D u,
\qquad x>0
\label{Leta}
\end{align}
which agrees with the representation
\begin{align}
\mathcal{L}  f(x)& = \frac{4}{\pi x}(PV)\!\!\int_0^\infty\left(f(xu)-f(x)\right)\frac{u^2}{(u^2-1)^2}\D u,
\qquad x>0,
\label{Leta2}
\end{align}
where $(PV)\!\int$ is understood as a principal value integral.
\end{cor}

\begin{proof}Because of  the arrangement of the characteristic exponent in \eqref{arrangement}, from \cite{CC06}, we know that its generator can be accordingly arranged to have action on $f\in \mathcal{C}^2(\mathbb{R}_{\geq 0})$ given by
\begin{align}
\mathcal{L}  f(x)& = \frac{\pi}{2}f'(x)+\frac{4}{\pi x}\int_0^\infty\left(f(xu)-f(x)-xf'(x)\log u\right)\frac{u^2}{(u^2-1)^2}\D u,
\qquad x>0.
\label{Leta}
\end{align}
For the second statement of the corollary, we need to show  that 
\begin{equation}
I: =(PV)\!\!\int_0^\infty \frac{u^2\log u}{(u^2-1)^2}\D u= \frac{\pi^2}{8}
\label{I}
\end{equation}
and that 
\[
(PV)\!\!\int_0^\infty\left(f(xu)-f(x)\right)\frac{u^2}{(u^2-1)^2}\D u
\]
is well defined. The latter is easily done on account of the fact that, near the singularity $u= 1$, $f(ux) - f(x) \approx (u-1)xf'(x) + O((u-1)^2)$,  $ x,u>0$, so that we can estimate the principal value of the integral there using partial fractions. 

To see why the  equality in \eqref{I} holds, note that after a change of variable $u ={\rm e}^{x}$ we see  
\begin{equation}
I =(PV)\!\! \int_{-\infty}^\infty \frac{x {\rm e}^x }{ ({\rm e}^x-{\rm e}^{-x})^2 } \D x = - (PV)\!\!\int_{-\infty}^\infty \frac{x {\rm e}^{-x} }{ ({\rm e}^x-{\rm e}^{-x})^2 } \D x,
\label{2I}
\end{equation}
where in the second equality  we have noted the simple change of variables $x\mapsto-x$.
It thus follows by adding the two integrals in \eqref{2I} together that 
\[
I = \frac{1}{2}  \int_{-\infty}^\infty \frac{x }{ ({\rm e}^x-{\rm e}^{-x}) } \D x = \frac{1}{2} \int_0^\infty \frac{x}{\sinh x} \D x=  \frac{\pi^2}{8}.
\]
%From a classic representation of the Riemann zeta function (cf. equation 3.523.1 of \cite{GR}), we know that,
%\[
%\zeta(s) = \frac{1}{2 \Gamma(s) (1-2^{-s})}  \int_0^\infty \frac{x^{s-1}}{\sinh x} \D x ,\qquad s\in\mathbb{C}, {\rm Re} (s) > 1.
%\]
%When $s=2$, using the special value $\zeta(2) = \pi^2/6$, this gives
%\[
%I = \frac{1}{2} \int_0^\infty \frac{x}{\sinh x} \D x =\frac{3}{4} \zeta(2)  \Gamma(2)  =  \pi^2/8,
%\]
%as required; 
where the final equality follows from equation 3.521.1 of \cite{GR}.  
%Alternatively, one may directly use the expansion
%\[
%x/sinh(x) = 2 x {\rm e}^{-x} (1-{\rm e}^{-2x})^{-1} = 2 x {\rm e}^{-1} (1+{\rm e}^{-2x}+{\rm e}^{-4x}+…)
%\]
%and integrate term by term, using monotone convergence.

Note, another way to approach the second part of the corollary is to use the standard definition of a Feller generator  on  $\mathcal{C}_c^\infty(\mathbb{R}_\geq 0)$, the space of compactly supported  smooth functions; cf \cite{KS19}. We have
\[
\mathcal{L}f (x) %= \lim_{t\to0}\frac{1}{t}\left(\texttt{E}_x[f(Y_t)] - f(x)\right) 
=\lim_{t\to0} \frac{1}{t}\left(
\int_0^\infty f(y)q^{(1),*}_t(x,y) - f(x)\right), \qquad x>0.
\]
Making use of \eqref{newSG} and monotone convergence, again taking note that the singularity in the integral can be dealt with in a similar manner, we see that
\begin{align}
\mathcal{L}f (x)& = \lim_{t\to0}\frac{4}{\pi}(PV)\!\!\int_0^\infty \left(f(y)-f(x)\right)\frac{y^2}{ (y^2-x^2)^2 +2t^2 (y^2 + x^2) + t^4 }\D y\notag \\
& = \frac{4}{\pi}(PV)\!\!\int_0^\infty \left(f(y)-f(x)\right)\frac{y^2}{ (y^2-x^2)^2 }\D y, \qquad x>0,
\label{firstL}
\end{align}
which agrees with \eqref{Leta2} after a simple change of variables.
\end{proof}

\subsection*{Connection to Cauchy process conditioned to stay positive}
It is also worthy of note in the general case $\alpha\in(0,2)$ that the process $Y$ does not agree with the law of a one-dimensional symmetric stable process conditioned to stay positive. 
The latter can be understood via the exact same limiting process in \eqref{condition}, again replacing the role of Brownian motion by that of the one-dimensional stable  process, inducing a new family of probabilities $(\texttt{P}^{1,1, \uparrow}_x, x>0)$ on $\mathbb{D}(\mathbb{R}_{\geq 0})$. Rather than corresponding to the change of measure \eqref{COM},the law of the Cauchy process conditioned to stay positive is related to that of the Cauchy process via   
 \begin{equation}
\left. \frac{\D \texttt{P}_x^{1,1, \uparrow}}{\D \texttt{P}^{1,1}_x}\right|_{\mathcal{F}_t} = \left(\frac{X_t}{x} \right)^{1/2}\mathbf{1}_{\{t<\tau^-_0(X)\}}, \qquad x>0, t\geq 0,
\label{COMuparrow}
 \end{equation}
 where  $\tau^-_0(X) = \inf\{t>0: X_t<0\}$.
 
\medskip

 There is nonetheless a close relationship between $(\texttt{P}_x,x>0)$ and $(\texttt{P}^{1,1, \uparrow}_x, x>0)$,
  which is best seen through the Lamperti representation \eqref{pssMp2}. Suppose we write $\Psi^\uparrow$ for the characteristic exponent of the L\'evy process that underlies the Cauchy process conditioned to stay positive. It is known from \cite{CC06}  (see also Chapter 5 of \cite{KPbook}) that 
\begin{equation}
\Psi^\uparrow(z) = \Psi(2z), \qquad z\in\mathbb{R}.
\label{2arg}
\end{equation}
If we write $\mu^\uparrow$ for the L\'evy measure associated to $\Psi^\uparrow$. This is equivalent to saying that $2\mu^\uparrow(x)= \mu(x/2)$, or indeed that the L\'evy process underlying the Cauchy process conditioned to stay positive is equal in law to $2\eta$. 
This is a curious relationship which is clearly related to the fact that the Doob $h$-transform in the definition \eqref{newSG} uses $h(x) = x$, whereas the Doob $h$-transform in \eqref{COMuparrow} uses $h(x) = \sqrt{x}$. It is less clear if or how this relationship extends to other values of $\alpha$.
From Lemma 2.2 in \cite{Pierre} we can now identify the following simple relationship.

\begin{cor} Denote by $Y^{\uparrow} = (Y^{\uparrow}_t, t\geq0)$ is the co-ordinate process of a  one-dimensional Cauchy process conditioned to stay positive. Then with $Y$ denoting the process in Theorem \ref{CauchyDM}, we have  space-time path transformation relating $Y$ to $Y^\uparrow$,
\[
(Y^\uparrow_t,t\geq0) \,\,\substack{\text{law}\\=} \,\,\Big((Y_{\chi(t)})^2, t\geq0\Big), \text{ where }\chi(t) = \inf\{s>0:\int_0^t Y^{-1}_u\D u>t\}, \qquad t\geq 0.
\]
\end{cor}

\subsection*{Pathwise representation}
One way to understand the Doob-McKean in the Cauchy setting is to consider it via a path transformation which mirrors the proof of Theorem \ref{CauchyDM}. Think of a two-dimensional Brownian motion $ \mathbb{P}^{(2)}$  on the $x$-$y$ plane which is stopped when hits the line $x = t$, that is at the time $\Gamma_t = \inf\{s>0 : \pi_x(\sqrt{2}B^{(2)}_s ) = t\}$, where $\pi_x$ is the projection of $\sqrt{2}B^{(2)}$ onto the $x$-axis. It is well known that $\Gamma_t$ is a \sfrac{1}{2}-stable subordinator and that 
$(\pi_y(\sqrt{2}B^{(2)}_{\Gamma_t}), t\geq 0)$ is  a Cauchy process where $\pi_y$ is the projection on to the $y$-axis. 

\medskip

Suppose now we replace $B^{(2)}$ by the $x$-$y$ planar process $(B, R)$, where $B$ is a one-dimensional Brownian motion and $R$ is an independent Bessel-3 process. Noting that $R$ is a Doob $h$-transform of $\pi_y(\sqrt{2}B^{(2)})$ killed on hitting the $x$-axis,  the independence of $B$ and $R$, and hence the independence of $(\Gamma_t, t\geq0)$ and $R$ means that the process $(\sqrt{2}R_{\Gamma_t}, t\geq0)$ agrees precisely with the transformation on the right-hand side of \eqref{newSG} with $\alpha = 1$.
\begin{figure}[h!]
\begin{center}
\includegraphics[width= 0.8\textwidth]{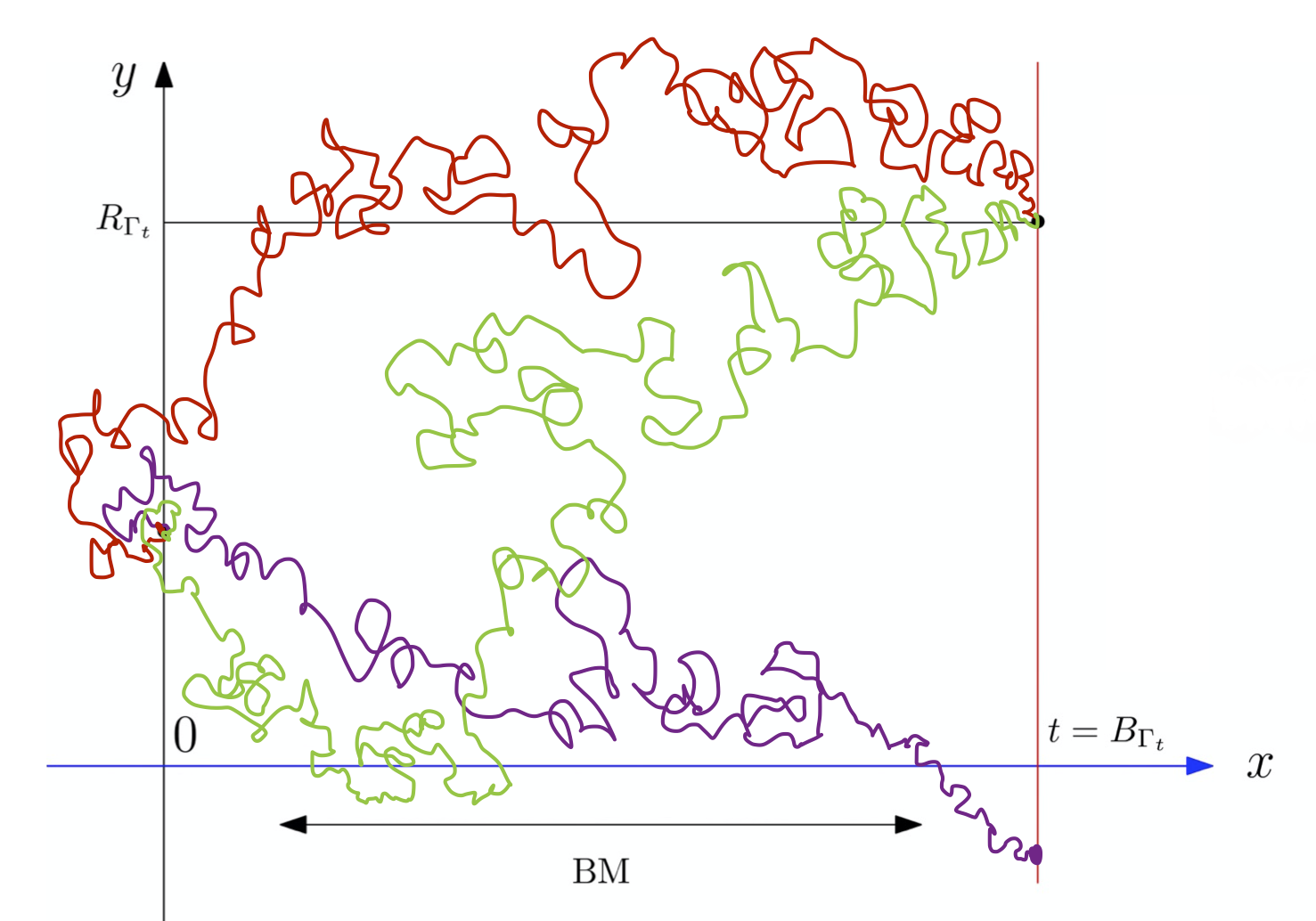}
\end{center}
\caption{\rm A pathwise representation of the Doob-McKean transformation for Cauchy processes. The red path depicts a sample path from the process $(B, R)$, where $B$ is a Brownian motion in the direction of the $x$ axis and $R$ is a Bessel-3 process in the direction of the $y$ axis, until it hits the vertical line $x = t$. The green and purple paths are sample paths from the the two dimensional Brownian motion  $B^{(2)}$ until first hitting of the vertical line $x = t$.} 
\label{BessBM}
\end{figure}

\subsection*{Generators}

We know that the generator of the process $Y$ in Theorem \ref{CauchyDM} is given by \eqref{Leta2}. The pathwise representation in the previous section, captured e.g. in Figure \ref{BessBM} also gives us some insight into the structure of the generator \eqref{Leta2}.

\medskip

As alluded to above, if $B$ is a one-dimensional Brownian motion, then $(\sqrt{2}B_{\Gamma_t}, t\geq 0)$ is a Cauchy process. Its generator $\mathcal{C}$ is written 
\begin{equation}
\mathcal{C}f(x) = \frac{1}{\pi}(PV)\int_{-\infty}^\infty \frac{f(y) - f(x)}{(y-x)^2}\dd y, \qquad f\in \mathcal{C}_c^\infty(\mathbb{R}_\geq 0).
\label{C}
\end{equation}
We want to connect the generator $\mathcal{C}$ with the processes $Y$ we see in the path decomposition, in particular withthe process $(R_{\Gamma_t}, t\geq0)$.

\medskip
 
We know from \eqref{transitions} that a Bessel-3 process is the result of Doob $h$-transforming the law of a Brownian motion killed on entry to $(-\infty,0)$. We have also seen e.g. in the proof of Theorem \ref{CauchyDM} that subordination with the \sfrac{1}{2}-stable process $(\Gamma_t, t\geq0)$ preserves the effect of the Doob $h$-transform. What we would like to understand is how the \sfrac{1}{2}-stable subordination of killed Brownian motion, i.e. $(q^{(1/2)}_t(x,y) - q^{(1/2)}_t(x,-y))$, plays out in \eqref{C}.

\medskip

To this end, we can think of jump rate  from $x\geq 0$ to $y\geq 0$ of the sub-Markov process with semigroup $(q^{(1/2)}_t(x,y) - q^{(1/2)}_t(x,-y))$,  as  being derived from a principal of `path counting' using jump rates of the Cauchy process. The generator of a Cauchy process killed on exiting the upper half line is given by 
\[
\mathcal{C}_+f(x) - \frac{1}{\pi x}\quad\text{ where }\quad\mathcal{C}_+f(x) : =  \frac{1}{\pi}\int_0^\infty \frac{f(y)-f(x)}{(y-x)^2}\dd y, \qquad f\in \mathcal{C}_c^\infty(\mathbb{R}_\geq 0).
\]
Indeed, the aforesaid process jumps from $x\geq 0$ to $y\geq 0$ at rate  $1/\pi(y-x)^2{\rm d}y$, however, we must subtract from this rate, the rate at which killing occurs by jumping from $x$ into the negative half line. The latter is 
\[
\frac{1}{\pi}\int_{-\infty}^0 \frac{1}{(y-x)^2}\dd y = \frac{1}{\pi}\int_{x}^\infty \frac{1}{z^2}\dd z = \frac{1}{\pi x}.
\]
%Roughly speaking, we can think of this as the result of 
%\[
%\lim_{t\to0}\frac{\texttt{E}^{1,1}_x[f(X_t)\mathbf{1}_{(X_t\geq 0)}] -f(x)}{t}
%\]
%
The combined effect of reflection principal and  \sfrac{1}{2}-stable subordination, suggests we must also subtract the rate at which jumps from $x\geq 0$ to $y\geq 0$ occur as the reflection of jumps from $x$ to $-y$, with the additional effect of killing on the lower half line, i.e. 
\[
\mathcal{C}_-f(x) - \frac{1}{\pi x}\quad\text{ where }\quad\mathcal{C}_-f(x) = \frac{1}{\pi}\int_0^\infty \frac{f(y)-f(x)}{(x+y)^2}\dd y , \qquad f\in \mathcal{C}_c^\infty(\mathbb{R}_\geq 0).
\]
We can thus identify the generator of $Y$, $\mathcal{L}$, as the following Doob $h$-transform.
\begin{lemma}We have
\[
\mathcal{L} f(x) = \frac{1}{x}\mathcal{D}(xf(x)), \qquad f\in \mathcal{C}_c^\infty(\mathbb{R}_\geq 0), x>0,
\]
where
\[
\mathcal{D} = \mathcal{C}_+-\mathcal{C}_- -\frac{2}{\pi x}.
\]
\end{lemma}
\begin{proof}
We compute (all integrals are Cauchy principal value integrals):
\begin{eqnarray*}
\frac1{x} \mathcal{D}(xf(x)) &=&
\frac1{\pi x} \int_0^\infty \frac{yf(y)-xf(x)}{(y-x)^2} dy
- \frac1{\pi x} \int_0^\infty \frac{yf(y)-xf(x)}{(y+x)^2} dy - \frac{2}{\pi x} f(x)\\
&=& \frac1{\pi x} \int_0^\infty \frac{4xy(yf(y)-xf(x))}{(y^2-x^2)^2} dy - \frac{2}{\pi x} f(x)\\
&=&  \frac{4}{\pi } \int_0^\infty \frac{y^2(f(y)-f(x))}{(y^2-x^2)^2} dy
+  \frac{4}{\pi } \int_0^\infty \frac{(y^2-xy)f(x)}{(y^2-x^2)^2} dy  - \frac{2}{\pi x} f(x)\\
&=& \mathcal{L} f(x),
\end{eqnarray*}
where the last identity follows from the definition of $\mathcal{L}$ and the fact that
$$(PV) \int_0^\infty \frac{2xy}{(y-x)(y+x)^2} dy =1.$$
\end{proof}
\medskip

%{\color{red}
Note that the `reflected' Cauchy process has generator $\mathcal{C}_R=\mathcal{C}_+ + \mathcal{C}_-$,
and we earlier identified  the  Cauchy process killed on going negative as having  generator $\mathcal{C}_A=\mathcal{C}_+ - 1/(\pi x)$.
These are related to the generator $\mathcal{D}$ via $\mathcal{C}_A=(\mathcal{D}+\mathcal{C}_R)/2$.
%We remark that we may also write
%$\mathcal{C}_A=(\mathcal{D}'+\mathcal{C}'_R)/2$, where 
%$\mathcal{D}'=\mathcal{D}+1/(2x)$ and $\mathcal{C}'_R=\mathcal{C}_R-1/(2x)$.
%The operators in this equation all have the same ground state $\varphi(x)=\sqrt{x}$, 
%that is, $\mathcal{C}_A\varphi=\mathcal{D}'\varphi =\mathcal{C}'_R\varphi =0$.
%}
The spectral problem associated with the Cauchy process on the half-line
with `reflecting' boundary is equivalent to the so-called `sloshing problem'
in the theory of linear water waves, and this has been extensively studied \cite{FK}.
The spectral problem associated with the Cauchy process on the half-line
with absorbing boundary conditions has been completely solved in \cite{KKMS}.

%The operator $H(d,\alpha,k)=(-\Delta)^{\alpha/2}+k/|x|$ is the Hamiltonian of the fractional
%Bohr atom [L], where $d$ is the dimension and $0<\alpha\le 2$ and $k$ are parameters.
%As such we may interpret the operator $\mathcal{C}_R’$ as the 
%radial part of $-H(1,1,1/2)$ and 
%$x^{-1}\circ\mathcal{D}’\circ x=\mathcal{L}+1/(2x)$ 
%as the radial part of $-H(3,1,-1/2)$.

%Is this correct?  I am assuming the generator of the isotropic Cauchy process is $-(-\Delta)^{\alpha/2}$
%(or is there a constant in front?) THIS IS THE WAY I WOULD DEFINE IT

\section{Concluding remarks}
Elliot and Feller \cite{EF} consider various examples of Cauchy processes constrained
to stay in a compact interval $[0,a]$.  One of the examples they consider (Example (d)
in their paper), has transition density
\begin{equation}\label{ef}
p_t(x,y) = \sum_{n=-\infty}^\infty [q_t(x,2an+y)-q_t(x,2an-y)],
\end{equation}
where $q_t(x,y)$ is the transition density of the one-dimensional Cauchy process.
They remark that \eqref{ef} defines {\em `a transition semi-group and determines a Markovian process,
but it is not the absorbing barrier process. [\ $\cdots$]\ It is not clear whether and how the process is
related to the Cauchy process.'}
In fact, the process considered in \cite{EF} is a Brownian motion in $[0,a]$ with Dirichlet
boundary conditions, time-changed by an independent stable subordinator of index $1/2$.
Moreover, it may be interpreted in terms of the Cauchy process via a similar pathwise interpretation
to the one outlined above for the half-line.

It is also natural to consider multi-dimensional versions.  For example, Dyson
Brownian motion is a Brownian motion in $\R^n$ conditioned never to
exit the Weyl chamber $C=\{x\in\R^n:\ x_1>\cdots>x_n\}$.  Its transition
density is given by
$$d_t(x,y) = h(x)^{-1} h(y) \sum_{\sigma\in S_n} \sgn(\sigma) p_t(x,\sigma y),$$
where the sum is over permutations, $\sigma y$ is the vector $y$ with
components permuted by $\sigma$, $h(x)=\prod_{i<j}(x_i-x_j)$ is the Vandermonde
determinant, and $p_t(x,y)$ is the standard Gaussian heat kernel in $\R^n$.
If we time-change this process by an independent stable subordinator of index $\alpha/2$,
and multiply by a factor of $\sqrt{2}$,
then the resulting process in $C$ has transition density
$$D_t(x,y) = h(x)^{-1} h(y) \sum_{\sigma\in S_n} \sgn(\sigma) P^{(\alpha)}_t(x,\sigma y),$$
where $P^{(\alpha)}_t(x, y)$ is the transition density of the isotropic $n$-dimensional stable
process with index $\alpha$.
We note that, in the case $\alpha=1$, this time-changed process may be interpreted as
the `radial part' of a $1$-Cauchy process in $\R^n$, as discussed in Section 5 of the paper \cite{ROSLER1998575}.

\section*{Acknowledgement}
Both authors would like to thank an anonymous referee for their remarks which lead to an improved version of this paper.
\bibliography{references}{}
\bibliographystyle{plain}

\end{document}